\documentclass[12pt, reqno,a4paper]{amsart}
\usepackage{enumerate, amssymb}
\usepackage[all]{xy}
\usepackage{rotating}
\usepackage[Symbol]{upgreek}

\usepackage{enumitem}
\usepackage{anysize}
\marginsize{1in}{1in}{0.8in}{1in}

\newtheorem{theorem}{Theorem}[section]
\newtheorem{proposition}[theorem]{Proposition}
\newtheorem{lemma}[theorem]{Lemma}

\theoremstyle{definition}
\newtheorem{definition}[theorem]{Definition}
\newtheorem{example}[theorem]{Example}

\theoremstyle{remark}

\numberwithin{equation}{section}
\def\sqr#1#2{{\,\vcenter{\vbox{\hrule height.#2pt\hbox{\vrule width.#2pt
height#1pt \kern#1pt\vrule width.#2pt}\hrule height.#2pt}}\,}}

\newcommand{\conv}{\operatorname{conv}}

\begin{document}
\title{Quantifying properties ($K$) and ($\mu^{s}$)}

\author{Dongyang Chen}
\address{School of Mathematical Sciences\\ Xiamen University,
Xiamen, 361005, China}
\email{cdy@xmu.edu.cn}

\author{Tomasz Kania}
\address{Mathematical Institute\\Czech Academy of Sciences\\\v Zitn\'a 25 \\115 67 Praha 1, Czech Republic\\
and\\  Institute of Mathematics and Computer Science\\ Jagiellonian University\\
{\L}ojasiewicza 6\\ 30-348 Krak\'{o}w, Poland}
\email{kania@math.cas.cz, tomasz.marcin.kania@gmail.com}

\author{Yingbin Ruan}
\address{College of Mathematics and informatics\\ Fujian Normal University,
Fuzhou, 350007, China}
\email{yingbinruan@sohu.com}

\thanks{{\em 2010 Mathematics Subject Classification:} 46B26; 46B50.\\
{\em Keywords:} Property ($K$); Grothendieck property; Reflexivity; Banach--Saks property; Property ($\mu^{s}$).\\
The first-named author was supported by the National Natural Science Foundation of China (Grant No.11971403) and the Natural Science Foundation of Fujian Province of China (Grant No. 2019J01024). The second-named author acknowledges with thanks funding received from SONATA 15 No.~2019/35/D/ST1/01734.
}

\begin{abstract}
A Banach space $X$ has \textit{property $(K)$}, whenever every weak* null sequence in the dual space admits a convex block subsequence $(f_{n})_{n=1}^\infty$ so that $\langle f_{n},x_{n}\rangle\to 0$ as $n\to \infty$ for every weakly null sequence $(x_{n})_{n=1}^\infty$ in $X$; $X$ has \textit{property $(\mu^{s})$} if every weak$^{*}$ null sequence in $X^{*}$ admits a subsequence so that all of its subsequences are Ces\`{a}ro convergent to $0$ with respect to the Mackey topology. Both property $(\mu^{s})$  and reflexivity (or even the Grothendieck property) imply property $(K)$. In the present paper we propose natural ways for quantifying the aforementioned properties in the spirit of recent results concerning other familiar properties of Banach spaces.
\end{abstract}

\date{Version: \today}
\maketitle

\baselineskip=18pt 	
\section{Introduction}
The present paper is inspired by many recent results that quantify various familiar properties of Banach spaces such as weak sequential completeness \cite{KPS}, reciprocal Dunford--Pettis property \cite{KS}, Schur property \cite{KalendaSupurny:2012}, Dunford--Pettis property \cite{KKS}, Banach--Saks property \cite{BKS}, property $(V)$ \cite{Krulisova:2017}, Grothendieck property \cite{Bendova:2014}, etc. We continue this line of research and investigate possible quantifications of related properties $(K)$ and ($\mu^{s}$) introduced by Kwapie\'n and Rodr\'{\i}guez, respectively.

Mazur's lemma (see, \emph{e.g.}, \cite[p.~11]{D}) states that every weakly convergent sequence in a Banach space has a convex block subsequence that is norm convergent to the same limit. (A sequence $(y_{n})_{n=1}^\infty$ in a Banach space $X$ is a \textit{convex block subsequence} of a~sequence $(x_{n})_{n=1}^\infty$ provided that there exists a strictly increasing sequence of positive integers $(k_{n})_{n=1}^\infty$ so that $y_{n}\in \conv(x_{i})_{i=k_{n-1}+1}^{k_{n}}$ for every $n\in \mathbb N$, where we set $k_{0}=0$; we denote by $\operatorname{cbs}((x_{n})_{n=1}^\infty )$ the collection of all convex block subsequences of $(x_{n})_{n=1}^\infty$.) Kalton and Pe{\l}czy\'{n}ski \cite[Proposition 2.2]{KP} proved that if a Banach space $X$ contains an isomorphic copy of $c_{0}$, then for every $\sigma$-finite measure $\mu$ the kernel of any surjection $Q$ from $L_1(\mu)$ onto $X$ is uncomplemented in its second dual. Consequently, $\ker Q$ is not isomorphic to a Banach lattice; the original argument relied on the  Lindenstrauss Lifting Principle. Having read a preliminary version of \cite{KP}, Kwapie\'{n} introduced property $(K)$ to provide an alternative proof of \cite[Propositon 2.2]{KP} which did not appeal to the Lindenstrauss Lifting Principle (Kwapie\'{n}'s idea was incorporated in \cite{KP}, where it was presented with his permission). Property $(K)$ is central to our considerations:
\begin{definition}\label{def:property K}
A Banach space $X$ has \textit{property $(K)$}, whenever every weak* null sequence in $X^{*}$ admits a convex block subsequence $(f_{n})_{n=1}^\infty$ so that $\lim\limits_{n\rightarrow \infty}\langle f_{n},x_{n}\rangle=0$ for every weakly null sequence $(x_{n})_{n=1}^\infty$ in $X$.
\end{definition}

Put equivalently, Definition~\ref{def:property K} stipulates that the sequence $(f_{n})_{n=1}^\infty$ converges to $0$ with respect to the Mackey topology $\mu(X^{*},X)$, which is the (locally convex) topology of uniform convergence in $X^*$ on weakly compact subsets of $X$ (see \cite[Lemma 3.5]{DDS}).\smallskip

Property $(K)$ may be thought of as a counterpart of Mazur's lemma with respect to the weak* topology. It was shown in \cite{KP} that the space $L_1(\mu)$ for a $\sigma$-finite measure has property ($K$), yet $c_{0}$ fails to have this property. \textit{Schur spaces}, \emph{i.e.}, spaces in which weak convergence of sequences coincides with norm convergence have property $(K)$ for trivial reasons. It follows from Mazur's lemma that Grothendieck spaces, in particular, reflexive spaces, have property $(K)$. (A Banach space $X$ is a~\textit{Gro\-then\-dieck space}, whenever every weak* convergent sequence in $X^{*}$ converges weakly.)

Figiel, Johnson, and Pe{\l}czy\'{n}ski \cite{FJP} refined property $(K)$ by introducing a weaker property that they call \textit{property $(k)$}; this property appeared implicitly also in \cite{Jo}. Property ($k$) was used in \cite{FJP} to show that the Separable Complementation Property need not pass to subspaces.

It was proved in \cite {FJP} that property $(k)$ is enjoyed by every separable subspace of a weakly sequentially complete Banach lattice, weakly sequentially complete Banach lattices with weak units, and every separable subspace of the predual of a von Neumann algebra. Oja \cite{FJP} pointed out that the Radon--Nikod\'{y}m property implies property $(k)$. However, it was shown \cite{FJP} that the $\ell_1$-sum of continuum many copies of $L_1[0,1]$ as well as Banach spaces containing complemented subspaces isomorphic to $c_{0}$ fail property $(k)$.

 Property $(K)$ admits a number of characterisations. More precisely, let $X$ be a~Banach space. Then the following assertions are equivalent:
\begin{enumerate}[label=(\alph*)]
\item $X$ has property $(K)$.
\item \label{K:b} Every weak$^{*}$ null sequence in $X^{*}$ admits a convex block subsequence $(f_{n})_{n=1}^\infty$ so that $\lim\limits_{n\rightarrow \infty}\langle f_{n},x_{n}\rangle=0$ for every weakly null sequence $(x_{n})_{n=1}^\infty $ in $X$.
\item\label{K:c} Every weak$^{*}$ null sequence in $X^{*}$ admits a convex block subsequence that is $\mu(X^{*},X)$-null.
\item\label{K:d} Every weak$^{*}$ convergent sequence in $X^{*}$ admits a convex block subsequence that is $\mu(X^{*},X)$-Cauchy.
\end{enumerate}

In Section~\ref{sect:quant} of the present paper, we prove quantitative versions of the aforestated characterisations. In order to do so, we introduce a~quantity $\alpha$ that characterises $\mu(X^{*},X)$-null sequences and subsequently we introduce a quantity $K_{1}$ that characterises property $(K)$. This is a quantitative version of clause \ref{K:c}. In order to quantify \ref{K:b}, we introduce a quantity $\beta$ that turns out to be equivalent to $\alpha$ for weak$^{*}$ null sequences. By using the quantity $\beta$, we introduce a further quantity $K_{2}$ that we then prove is then equivalent to the quantity $K_{1}$. Finally, to quantify \ref{K:d}, we introduce a quantity $K_{3}$ in terms of $\operatorname{ca}_{\rho^{*}}$ defined in \cite{KKS} that measures $\mu(X^{*},X)$-Cauchyness and prove that $K_{3}$ is equivalent to $K_{1}$. In summary, we quantify property $(K)$ by means of the following estimates:
$$K_{2}(X)\leqslant K_{1}(X)\leqslant 2K_{2}(X)$$ and $$K_{1}(X)\leqslant K_{3}(X)\leqslant 4K_{1}(X).$$

Furthermore, we investigate the values of the quantity $K_{2}$ in certain familiar Banach spaces failing property $(K)$ and obtain that, in particular, $K_{2}(c_{0})=1$ and the $K_{2}$-value of the $\ell_1$-sum of $\mathfrak{c}$ copies of $L_1[0,1]$ is equal to $1$. (Curiously, Frankiewicz and Plebanek \cite{FP} proved that under Martin's Axiom, the $\ell_1$-sum of fewer than $\mathfrak{c}$ copies of $L_1[0,1]$ still has property $(K)$.)\smallskip

The purpose of Section~\ref{sect:groth} is to quantify the following widely known implications:
\begin{equation}
\text{$X$ is reflexive $\Rightarrow$ $X$ is a Grothendieck space $\Rightarrow$ $X$ has property $(K)$.}
\tag{$\star$}
\end{equation}

In order to quantify ($\star$), we first make a slight improvement on characterisations of weak compactness due to \"{U}lger \cite{U} (see also \cite{DRS}). Using this improvement, we establish a characterisation of the Grothendieck property, which is used to introduce a quantity $G$ measuring the Grothendieck property. This quantification of the Grothendieck property is different from the quantitative Grothendieck property proposed by Kruli\v{s}ov\'a (ne\'e Bendov\'a) in \cite{Bendova:2014}.
Again using the improvement, we introduce a new quantity $R$ measuring reflexivity for Banach spaces. Meanwhile, the relationship between the quantity $R$ and several classical equivalent quantities measuring weak non-compactness is discussed. We also investigate possible values of the quantity $R$ of some classical Banach spaces. Having introduced $G$ and $R$, we quantify the implications ($\star$) as follows: $$K_{1}(X)\leqslant G(X)\leqslant R(X^{*}).$$
Avil\'{e}s and Rodr\'{\i}guez \cite{AR} studied the implications ($\star$) for Banach spaces not containing isomorphic copies of $\ell_1$ and proved that for such space $X$:
\begin{equation}
\text{$X$ is reflexive $\Leftrightarrow$ $X$ is a Grothendieck space $\Leftrightarrow$ $X$ has property $(K)$.}
\tag{$\star\star$}
\end{equation}
Finally, we quantify ($\star\star$) as follows: $$K_{1}(X)\leqslant G(X)\leqslant R(X^{*})\leqslant 8K_{2}(X).$$\medskip

A bounded subset $A$ of a Banach space $X$ is a \textit{Banach--Saks set} if each sequence in $A$ has a Ces\`{a}ro convergent subsequence. A Banach space $X$ is said to have the \textit{Banach--Saks property} if its closed unit ball $B_{X}$ is a Banach--Saks set. Banach and Saks proved in \cite{BS} that the spaces $L_p[0,1]$ and $\ell_p$ ($1<p<\infty$) enjoy the Banach--Saks property, hence the name. Kakutani \cite{Ka} later showed that uniformly convex spaces have the Banach--Saks property (hence so do superreflexive spaces). Any space with the Banach--Saks property is reflexive \cite{NW}, but there are reflexive spaces without the Banach--Saks property \cite{Ba}. A localised version of the result of \cite{NW} says that any Banach--Saks set is relatively weakly compact \cite{LART}.\smallskip

It follows from the Erd\H{o}s--Magidor Theorem (Theorem~\ref{1.1}) that a Banach space $X$ has the Banach--Saks property if and only if every bounded sequence in $X$ admits a~subsequence such that all of its subsequences are Ces\`{a}ro convergent.
Property $(\mu^{s})$, introduced by Rodr\'{\i}guez \cite{Ro}, is a statement refining the Banach--Saks property for weak* null sequences in the dual space.

\begin{definition}A Banach space $X$  has \textit{property $(\mu^{s})$}, whenever every weak$^{*}$ null sequence in $X^{*}$ admits a subsequence so that all of its subsequences are Ces\`{a}ro convergent to $0$ with respect to $\mu(X^{*},X)$.\end{definition}
Clearly, if $X^{*}$ has the Banach--Saks property, then $X$ has property $(\mu^{s})$. The converse is true for reflexive spaces \cite[Proposition 2.2]{Ro}. Moreover, it was pointed out \cite[Lemma 2.1, Remark 2.3]{Ro} that property $(\mu^{s})$ is strictly stronger than property ($K$).

The goal of Section~\ref{sect:mus} is to quantify the following implications (\cite{Ro}):
\begin{equation}
\text{$X^{*}$ has the Banach--Saks property $\Rightarrow$ $X$ has property $(\mu^{s})$ $\Rightarrow$ $X$ has property $(K)$.}
\tag{$\star\star\star$}
\end{equation}

To quantify property $(\mu^{s})$, we first introduce a quantity $c\alpha$ by means of $\alpha$ that measures the rate of Ces\`{a}ro convergence to $0$ with respect to $\mu(X^{*},X)$. By using the quantity $c\alpha$, we introduce a quantity $\mu^{s}$ that we then prove characterises property $(\mu^{s})$. Furthermore, we introduce a quantity $\operatorname{bs}(X)$ that characterises the Banach--Saks property of a Banach space $X$. This quantity is stronger than the quantity introduced in \cite{BKS} that measures how far a bounded set is from being Banach--Saks. By using the quantities $\mu^{s}$ and $\operatorname{bs}$, we quantify ($\star\star\star$) as follows:
$$\frac{1}{3}K_{1}(X)\leqslant \mu^{s}(X)\leqslant \operatorname{bs}(X^{*}).$$
Finally, we prove that, for a reflexive space $X$, $$\mu^{s}(X)\leqslant \operatorname{bs}(X^{*})\leqslant 4\mu^{s}(X),$$ which is a quantitative version of \cite[Proposition 2.2]{Ro}.

\section{Preliminaries} We use standard notation and terminology in-line with \cite{AK} and \cite{LT}. Throughout this paper, all Banach spaces are infinite-dimensional over the fixed field of real or complex numbers. By a \emph{subspace} we mean a closed, linear subspace. An \emph{operator} will always mean a bounded linear operator. If $X$ is a Banach space, we denote by $B_{X}$ the closed unit ball $\{x\in X\colon \|x\|\leqslant 1\}$ and by $\mathcal{F}_{X}$ the family of all weakly compact subsets in $B_{X}$. For a subset $A$ of $X$, $\conv(A)$ stands for the convex hull of $A$. For brevity of notation, we denote by $\textrm{ss}((x_{n})_{n=1}^\infty )$ the family of all subsequences of a sequence $(x_{n})_{n=1}^\infty$.

\subsection{Weak compactness} Let us invoke the following characterisation of weak compactness due to \"{U}lger \cite{U}.

\begin{lemma}\label{3.1}
A bounded subset $A$ of a Banach space $X$ is relatively weakly compact if and only if given any sequence $(x_{n})_{n=1}^\infty $ in $A$, there exists a sequence $(z_{n})_{n=1}^\infty $ with $z_{n}\in \conv(x_{i}:i\geqslant n)$ that converges weakly.
\end{lemma}

Diestel, Ruess, and Schachermayer \cite{DRS} improved Lemma \ref{3.1} as follows.

\begin{lemma}\label{3.2}
For a bounded subset $A$ of $X$ the following statements are equivalent:
\begin{enumerate}
\item $A$ is relatively weakly compact.
\item For every sequence $(x_{n})_{n=1}^\infty $ in $A$, there is a norm-convergent sequence $(z_{n})_{n=1}^\infty $ such that $z_{n}\in \conv(x_{i}\colon i\geqslant n)$.
\item For every sequence $(x_{n})_{n=1}^\infty $ in $A$, there is a weakly convergent sequence $(z_{n})_{n=1}^\infty $ such that $z_{n}\in \conv(x_{i}\colon i\geqslant n)$.
\end{enumerate}
\end{lemma}

Let $A$ and $B$ are non-empty subsets of a Banach space $X$, we set
\begin{itemize}
    \item $\textrm{d}(A,B)=\inf\{\|a-b\|\colon a\in A,b\in B\},$
    \item $\widehat{\textrm{d}}(A,B)=\sup\{\textrm{d}(a,B)\colon a\in A\}.$
\end{itemize}
$\textrm{d}(A,B)$ is the ordinary distance between $A$ and $B$, and $\widehat{\textrm{d}}(A,B)$ is the (non-symmetrised) Hausdorff distance from $A$ to $B$. When $A$ is a bounded subset of a Banach space $X$, following \cite{KKS}, we set
\begin{itemize}
\item $\textrm{wk}_{X}(A)=\widehat{\textrm{d}}\big(\overline{A}^{\sigma(X^{**},X^{*})},X\big);$
\item $\operatorname{wck}_{X}(A)=\sup\{\textrm{d}(\textrm{clust}_{X^{**}}((x_{n})_{n=1}^\infty ),X)\colon (x_{n})_{n=1}^\infty $ is a sequence in $A\}$,\\ where $\textrm{clust}_{X^{**}}((x_{n})_{n=1}^\infty )$ is the set of all weak$^{*}$-cluster points of $(x_{n})_{n=1}^\infty $ in $X^{**}$.
\item $\upgamma_{X}(A)=\sup\{|\lim\limits_{n}\lim\limits_{m}\langle f_{m},x_{n}\rangle-\lim\limits_{m}\lim\limits_{n}\langle f_{m},x_{n}\rangle|\colon (x_{n})_{n=1}^\infty $ is a sequence in $A$, $(f_{m})_{m=1}^\infty$ is a sequence in $B_{X^{*}}$ and all the involved limits exist$\}$.
\end{itemize}
It follows from \cite[Theorem 2.3]{AC} that
$$\operatorname{wck}_{X}(A)\leqslant \textrm{wk}_{X}(A)\leqslant \upgamma_{X}(A)\leqslant 2\operatorname{wck}_{X}(A).$$
\subsection{Mackey topology}
 Let $X$ be a Banach space. The Mackey topology, $\mu(X^{*},X)$, is the strongest locally convex topology on $X^{*}$ which is compatible with the dual pairing $\langle X^{*},X\rangle$. In particular, $\overline{C}^{w^{*}}=\overline{C}^{\mu(X^{*},X)}$ for every convex subset $C$ of $X^{*}$. If the dual unit ball endowed with the relative Mackey topology, $(B_{X^{*}},\mu(X^{*},X))$, is metrisable, then $X$ has property $(K)$. Schl\"{u}chtermann and Wheeler \cite{SW} term Banach spaces $X$ for which $(B_{X^{*}},\mu(X^{*},X))$ is metrisable as \textit{strongly weakly compactly generated} (SWCG) spaces. As proved in \cite{SW}, a Banach space $X$ is SWCG if and only if there exists a~weakly compact subset $K$ of $X$ so that for every weakly compact subset $L$ of $X$ and $\varepsilon>0$, there is a positive integer $n$ with $L\subseteq nK+\varepsilon B_{X}$. Moreover, reflexive spaces, separable Schur spaces, the space of operators of trace-class on a separable Hilbert space, and $L_1(\mu)$ for a $\sigma$-finite measure $\mu$ are SWCG.

\subsection{Banach--Saks sets} Let $(x_{n})_{n=1}^\infty $ be a bounded sequence in a Banach space. We set $$\operatorname{ca}((x_{n})_{n=1}^\infty )=\inf_{n\in\mathbb N}\sup_{k,l\geqslant n}\|x_{k}-x_{l}\|.$$
Clearly $(x_{n})_{n=1}^\infty $ is norm-Cauchy if and only if $\operatorname{ca}((x_{n})_{n=1}^\infty )=0$.
Following \cite{BKS}, we define $$\textrm{cca}\big((x_{n})_{n=1}^\infty \big)=\operatorname{ca}\big((\frac{1}{n}\sum_{i=1}^{n}x_{i})_{n=1}^\infty\big).$$
Clearly, $\textrm{cca}((x_{n})_{n=1}^\infty )=0$ if and only if $(x_{n})_{n=1}^\infty $ is Ces\`{a}ro convergent.

A subset $A$ of a Banach space is \emph{Banach--Saks}, whenever every sequence in $A$ has a~Ces\`aro convergent subsequence.

We shall require the well-known $0$-$1$ law by Erd\H{o}s--Magidor \cite{EM}.

\begin{theorem}(Erd\H{o}s--Magidor)\label{1.1}
Every bounded sequence in a Banach space has a subsequence such that either all its further subsequences are Ces\`{a}ro convergent, or none of them.
\end{theorem}

For a bounded set $A$ in a Banach space $X$, Kruli\v{s}ov\'a (ne\'e Bendov\'{a}), Kalenda, and Spurn\'{y} \cite{BKS} introduced the following quantity
$$\operatorname{bs}(A)=\sup_{(x_{n})_{n=1}^\infty \subseteq A}\inf_{(y_{n})_{n=1}^\infty \in \textrm{ss}((x_{n})_{n=1}^\infty )}\textrm{cca}((y_{n})_{n=1}^\infty )$$
measuring how far is $A$ from being a Banach--Saks set. More precisely, they proved that $A$ is a Banach--Saks set if and only if $\operatorname{bs}(A)=0$.


\section{Quantifications of property ($K$)}\label{sect:quant}
Let $(f_{n})_{n=1}^\infty$ be a bounded sequence in $X^{*}$. Following \cite{KKS}, we set
$$\operatorname{ca}_{\rho^{*}}((f_{n})_{n=1}^\infty)=\sup_{K\in \mathcal{F}_{X}}\inf_{n\in \mathbb N}\sup_{k,l\geqslant n}\sup_{x\in K}|\langle f_{k}-f_{l},x\rangle|,$$
then $\operatorname{ca}_{\rho^{*}}((f_{n})_{n=1}^\infty)=0$ if and only if $(f_{n})_{n=1}^\infty$ is $\mu(X^{*},X)$-Cauchy (\emph{i.e.}, $\mu(X^{*},X)$-convergent as the Mackey topology $\mu(X^{*},X)$ \cite[Proposition 4 on p.~197]{Ja} is complete). We set $$\alpha((f_{n})_{n=1}^\infty)=\sup_{K\in \mathcal{F}_{X}}\limsup_{n\to \infty}\sup_{x\in K}|\langle f_{n},x\rangle|,$$
then $\alpha((f_{n})_{n=1}^\infty )=0$ if and only if $(f_{n})_{n=1}^\infty$ is $\mu(X^{*},X)$-null, and
$$\beta((f_{n})_{n=1}^\infty)=\sup_{(x_{n})_{n=1}^\infty \subseteq B_{X} \atop \textrm{weakly null}}\limsup_{n\to \infty}|\langle f_{n},x_{n}\rangle|.$$
The following result is a quantitative version of \cite[Lemma 3.3]{DDS}.

\begin{lemma}\label{2.1}
Let $(f_{n})_{n=1}^\infty$ be a weak$^{*}$ null sequence in $X^{*}$. Then
$$\beta((f_{n})_{n=1}^\infty)\leqslant \alpha((f_{n})_{n=1}^\infty)\leqslant 2\beta((f_{n})_{n=1}^\infty).$$
\end{lemma}
\begin{proof}
The former inequality is trivial. It remains to prove the latter one.

Let $0<c<\alpha\big((f_{n})_{n=1}^\infty\big)$. Then there exist a weakly compact subset $K\subseteq B_{X}$, a~subsequence $(f_{k_{n}})_{n=1}^\infty$ of $(f_{n})_{n=1}^\infty$, and a sequence $(x_{n})_{n=1}^\infty $ in $K$ so that $|\langle f_{k_{n}},x_{n}\rangle|>c$ for all $n$. Since $K$ is weakly compact, by the Eberlein--\v{S}mulian theorem, $(x_{n})_{n=1}^\infty $ admits a subsequence $(x_{n_{m}})_{m=1}^\infty$ that converges weakly to some $x\in K$. We define a sequence $(z_{n})_{n=1}^\infty $ in $X$ by $$z_{k_{n_{m}}}=\frac{1}{2}(x_{n_{m}}-x)\quad (m=1,2,\ldots)$$ and $z_{n}=0$ for $n\notin \{k_{n_{m}}\}_{m=1}^{\infty}$. Then $(z_{n})_{n=1}^\infty $ is weakly null in $B_{X}$. For each $m$, we get
$$|\langle f_{k_{n_{m}}},z_{k_{n_{m}}}\rangle|=\frac{1}{2}|\langle f_{k_{n_{m}}},x_{n_{m}}\rangle-\langle f_{k_{n_{m}}},x\rangle|
\geqslant \frac{1}{2}(c-|\langle f_{k_{n_{m}}},x\rangle|).$$
Since $(f_{n})_{n=1}^\infty$ is $\sigma(X^{*},X)$-null, we get
$$\limsup_{n\to \infty}|\langle f_{n},z_{n}\rangle|\geqslant \limsup_{m\to \infty}|\langle f_{k_{n_{m}}},z_{k_{n_{m}}}\rangle|\geqslant \frac{c}{2}.$$
Since $c$ is arbitrary, we arrive at $\beta((f_{n})_{n=1}^\infty)\geqslant\frac{1}{2}\alpha((f_{n})_{n=1}^\infty).$\end{proof}

\begin{lemma}\label{2.2}
Let $(f_{n})_{n=1}^\infty$ be a bounded sequence in $X^{*}$ and $f\in X^{*}$. Then $$\operatorname{ca}_{\rho^{*}}((f_{n})_{n=1}^\infty)\leqslant 2\alpha((f_{n}-f)_{n=1}^\infty).$$
\end{lemma}
\begin{proof}
Let $c>\alpha((f_{n}-f)_{n=1}^\infty)$ be arbitrary. Let $K\in \mathcal{F}_{X}$. Then there exists a positive integer $n$ so that $\sup\limits_{x\in K}|\langle f_{k}-f,x\rangle|<c$ for all $k\geqslant n$. Hence, for $k,l\geqslant n$, we get
$$\sup\limits_{x\in K}|\langle f_{k}-f_{l},x\rangle|=\sup\limits_{x\in K}|\langle (f_{k}-f)-(f_{l}-f),x\rangle|\leqslant 2c.$$
This implies that $\operatorname{ca}_{\rho^{*}}((f_{n})_{n=1}^\infty)\leqslant 2c$. As $c$ was arbitrary, the proof is complete.
\end{proof}

\begin{lemma}\label{2.3}
Suppose that $(f_{n})_{n=1}^\infty$ converges  to $f\in X^{*}$ in the weak* topology. Then $$\alpha((f_{n}-f)_{n=1}^\infty)\leqslant \operatorname{ca}_{\rho^{*}}((f_{n})_{n=1}^\infty).$$
\end{lemma}
\begin{proof}
Let $c>\operatorname{ca}_{\rho^{*}}((f_{n})_{n=1}^\infty )$ be arbitrary. Let $K\in \mathcal{F}_{X}$. Then there exists a positive integer $n$ so that $\sup\limits_{x\in K}|\langle f_{k}-f_{l},x\rangle|<c$ for all $k,l\geqslant n$. Hence, for each $x\in K$, we get $|\langle f_{k}-f_{l},x\rangle|<c$ for all $k,l\geqslant n$.
Letting $l\rightarrow \infty$, we get $|\langle f_{k}-f,x\rangle|\leqslant c$. This means that $\sup\limits_{x\in K}|\langle f_{k}-f,x\rangle|\leqslant c$ for all $k\geqslant n$ and, consequently, $\alpha\big((f_{n}-f)_{n=1}^\infty\big)\leqslant c$. As $c$ was arbitrary, the proof is finished.\end{proof}

\begin{definition}
Let $X$ be a Banach space. We set
$$K_{1}(X)=\sup_{(f_{n})_{n=1}^\infty \subseteq B_{X^{*}} \atop \operatorname{weak}^{*}\textrm{null}}\inf_{(g_{n})_{n=1}^\infty \in \operatorname{cbs}((f_{n})_{n=1}^\infty )}\alpha\big((g_{n})_{n=1}^\infty\big),$$

$$K_{2}(X)=\sup_{(f_{n})_{n=1}^\infty \subseteq B_{X^{*}} \atop \operatorname{weak}^{*}\textrm{null}}\inf_{(g_{n})_{n=1}^\infty \in \operatorname{cbs}((f_{n})_{n=1}^\infty )}\beta\big((g_{n})_{n=1}^\infty\big),$$

and

$$K_{3}(X)=\sup_{(f_{n})_{n=1}^\infty \subseteq B_{X^{*}} \atop \operatorname{weak}^{*}\textrm{Cauchy}}\inf_{(g_{n})_{n=1}^\infty \in \operatorname{cbs}((f_{n})_{n=1}^\infty )}\operatorname{ca}_{\rho^{*}}\big((g_{n})_{n=1}^\infty \big).$$
\end{definition}
The three quantities $K_{1},K_{2}$, and $K_{3}$ are actually equivalent.

\begin{proposition}\label{prop:K123}
Let $X$ be a Banach space. Then
\begin{enumerate}[label=(\roman*)]
\item\label{prop:K123:i} $K_{2}(X)\leqslant K_{1}(X)\leqslant 2K_{2}(X),$
\item\label{prop:K123:ii} $K_{1}(X)\leqslant K_{3}(X)\leqslant 4K_{1}(X).$
\end{enumerate}
\end{proposition}
\begin{proof}
The statement \ref{prop:K123:i} follows from Lemma \ref{2.1}. It suffices to prove \ref{prop:K123:ii}.

It follows from Lemma \ref{2.3} that $K_{1}(X)\leqslant K_{3}(X)$.
Let $0<c<K_{3}(X)$. Then there exists a weak$^{*}$-Cauchy sequence $(f_{n})_{n=1}^\infty$ in $B_{X^{*}}$ so that for every $
(h_{n})_{n=1}^\infty\in \operatorname{cbs}((f_{n})_{n=1}^\infty )$ we have $\operatorname{ca}_{\rho^{*}}((h_{n})_{n=1}^\infty )>c$. Clearly, $(f_{n})_{n=1}^\infty$ converges to some $f\in B_{X^{*}}$ in the weak* topology. Take any
$(g_{n})_{n=1}^\infty \in \operatorname{cbs}((\frac{1}{2}(f_{n}-f))_{n=1}^\infty)$. Then $2g_{n}=h_{n}-f$ ($n\in \mathbb N$), where $(h_{n})_{n=1}^\infty\in
\operatorname{cbs}((f_{n})_{n=1}^\infty )$. By Lemma \ref{2.2}, we get
$$c<\operatorname{ca}_{\rho^{*}}\big((h_{n})_{n=1}^\infty \big)=2\operatorname{ca}_{\rho^{*}}\big((g_{n})_{n=1}^\infty \big)\leqslant 4\alpha\big((g_{n})_{n=1}^\infty\big).$$
Hence $c\leqslant 4K_{1}(X)$. Since $c$ is arbitrary, we get $K_{3}(X)\leqslant 4K_{1}(X).$\end{proof}

The subsequent result implies that the quantities $K_{1},K_{2}$, and $K_{3}$ do characterise property $(K)$.

\begin{theorem}\label{2.4}
A Banach space $X$ has property $(K)$ if and only if $K_{1}(X)=0$.
\end{theorem}

To prove Theorem \ref{2.4}, we require two elementary lemmata whose proofs are omitted.

\begin{lemma}\label{2.5}
If $(y_{n})_{n=1}^\infty \in \operatorname{cbs}((x_{n})_{n=1}^\infty )$, then $\operatorname{cbs}((y_{n})_{n=1}^\infty )\subseteq \operatorname{cbs}((x_{n})_{n=1}^\infty )$. More precisely, if
$y_{n}\in \conv(x_{i})_{i=k_{n-1}+1}^{k_{n}}$ and $z_{n}\in \conv(y_{j})_{j=m_{n-1}+1}^{m_{n}}$, then $z_{n}\in \conv
(x_{i})_{i=k_{m_{n-1}}+1}^{k_{m_{n}}}.$
\end{lemma}

\begin{lemma}\label{2.6}
Let $(f_{n})_{n=1}^\infty$ be a bounded sequence in $X^{*}$. Then
$$\alpha\big((g_{n})_{n=1}^\infty\big)\leqslant \alpha\big((f_{n})_{n=1}^\infty \big), \quad \big((g_{n})_{n=1}^\infty \in \operatorname{cbs}((f_{n})_{n=1}^\infty )\big).$$
\end{lemma}

\begin{proof}[Proof of Theorem \ref{2.4}] The necessity is trivial, so it suffices to prove the sufficiency.\smallskip

Let $(f_{n})_{n=1}^\infty$ be a weak*-null sequence in $B_{X^{*}}$. Since $K_{1}(X)=0$, for each $k$ we get inductively a sequence
$(f^{(k)}_{n})_{n=1}^\infty $ in $X^{*}$ so that for  $k\in \mathbb N.$
\begin{itemize}
\item $(f^{(1)}_{n})_{n=1}^\infty\in \operatorname{cbs}((f_{n})_{n=1}^\infty ),$
\item $(f^{(k+1)}_{n})_{n=1}^\infty\in \operatorname{cbs}((f^{(k)}_{n})_{n=1}^\infty ),$
\item $\alpha((f^{(k)}_{n})_{n=1}^\infty )<\frac{1}{k}.$
\end{itemize}
For $n\in\mathbb N$, we set $g_{n}=f^{(n)}_{n}$. By Lemma \ref{2.5}, we get $(g_{n})_{n\geqslant k}\in \operatorname{cbs}((f^{(k)}_{n})_{n=1}^\infty )$ for each $k$.
By Lemma \ref{2.6}, we get $$\alpha\big((g_{n})_{n=1}^\infty\big)=\alpha\big((g_{n})_{n\geqslant k}\big)\leqslant \alpha\big((f^{(k)}_{n})_{n=1}^\infty \big)<\tfrac{1}{k} \quad (k=1,2,\ldots).$$
This means that $\alpha\big((g_{n})_{n=1}^\infty\big)=0$ and $(g_{n})_{n=1}^\infty$ is $\mu(X^{*},X)$-null. Consequently, $X$ has property $(K)$.\end{proof}

\begin{example}${}$
\begin{enumerate}[label=(\alph*)]
\item\label{example:a} Let $X$ be a Banach space so that $B_{X^{*}}$ is $\sigma(X^{*},X)$-sequentially compact or  $\ell_1$ does not embed into $X$. If $X$ contains a~subspace isomorphic to $c_{0}$, then $K_{2}(X)=1$. In particular, $$K_{2}(c_{0})=K_{2}(c)=K_{2}(C[0,1])=1.$$
\item\label{example:b} $K_{2}\big(\ell_1({\mathbb{R}}, L_1[0,1])\big) = 1;$ here $\ell_1({\mathbb{R}}, L_1[0,1])$ stands for the $\ell_1$-sum of $\mathfrak{c}$ copies of $L_1[0,1]$.
\end{enumerate}
\end{example}
\begin{proof}
\ref{example:a}. Let $\varepsilon>0$. It follows from \cite[Theorem 6]{DRT} (respectively, \cite[Theorem 2.2]{DF}) that there exists a subspace $Z$ of $X$ so that $Z$ is $(1+\varepsilon)$-isomorphic to $c_{0}$ and a~projection $P$ from $X$ onto $Z$ with $\|P\|\leqslant 1+\varepsilon$. Let $T\colon c_{0}\rightarrow Z$ be an operator so that
$$\frac{1}{1+\varepsilon}\|z\|\leqslant \|Tz\|\leqslant \|z\| \quad (z\in c_{0}).$$
Let $S=T^{-1}P$. Then $ST=I_{c_{0}}$ and $\|S\|\leqslant (1+\varepsilon)^{2}$. For each $n$, we set $f_{n}=\frac{S^{*}e^{*}_{n}}{(1+\varepsilon)^{2}}$, where $(e^{*}_{n})_{n}$ is the unit vector basis of $\ell_1$. Then $(f_{n})_{n=1}^\infty$ is weak$^{*}$ null in $B_{X^{*}}$. Take any $(y^{*}_{n})_{n=1}^\infty\in \operatorname{cbs}((f_{n})_{n=1}^\infty )$ and write $$y^{*}_{n}=\sum\limits_{i=k_{n-1}+1}^{k_{n}}\lambda_{i}f_{i},$$
where $\sum\limits_{i=k_{n-1}+1}^{k_{n}}\lambda_{i}=1,$ and $\lambda_{i}\geqslant 0.$ For every $n$, let $z_{n}=\sum\limits_{i=k_{n-1}+1}^{k_{n}}e_{i},$ where $(e_{n})_{n=1}^\infty$ is the unit vector unit basis of $c_{0}$. Clearly, $(Tz_{n})_{n=1}^\infty$ is weakly null in $B_{X}$. Moreover, for every $n$, we get
$$|\langle y^{*}_{n},Tz_{n}\rangle|=\frac{1}{(1+\varepsilon)^{2}}|\langle \sum\limits_{i=k_{n-1}+1}^{k_{n}}\lambda_{i}e^{*}_{i},\sum\limits_{j=k_{n-1}+1}^{k_{n}}e_{j}\rangle|
=\frac{1}{(1+\varepsilon)^{2}}.$$
This means that $\beta((g_{n})_{n=1}^\infty)\geqslant \frac{1}{(1+\varepsilon)^{2}}$ and so $K_{2}(X)\geqslant \frac{1}{(1+\varepsilon)^{2}}$. Letting $\varepsilon\rightarrow 0$, we get $K_{2}(X)=1$.

\ref{example:b}. Let $\Lambda$ be the set of all strictly increasing sequences $(k_{n})_{n=1}^\infty$ of positive integers with $k_{1}=1$. Set $X=\ell_1({\Lambda}, L_1[0,1]).$ Let $(r_{j})_{j=1}^\infty$ be a sequence of Rademacher functions. Define $(g^{*}_{n})_{n=1}^\infty\subseteq X^{*}$ by $$g^{*}_{n}(t)=r_{j(n,t)},$$ where $t=(k_{m})_{m=1}^\infty\in\Lambda$, $k_{j(n,t)}\leqslant n<k_{j(n,t)+1}.$

Since $g^{*}_{n}(t)\stackrel{\operatorname{weak}^{*}}{\longrightarrow}0$ in $L_{\infty}[0,1]$ ($t\in \Lambda$) and $\|g^{*}_{n}\|=1$ ($n\in \mathbb N$), we get $g^{*}_{n}\stackrel{\operatorname{weak}^{*}}{\longrightarrow}0$. Given $(h^{*}_{m})_{m=1}^\infty\in \operatorname{cbs}((g^{*}_{n})_{n})$, we write  $$h^{*}_{m}=\sum\limits_{j=k^{\circ}_{m}}^{k^{\circ}_{m+1}-1}\lambda_{j}g^{*}_{j}\quad (t_{0}=(k^{\circ}_{m})_{m}\in \Lambda).$$ For each $m$, define $h_{m}\in X$ by $h_{m}(t)=r_{m}$ if $t=t_{0}$ and $h_{m}(t)=0$ otherwise. Then $(h_{m})_{m=1}^\infty$ is weakly null in $B_{X}$. Moreover, $\langle h^{*}_{m},h_{m}\rangle=1$ for each $m$.
This implies that $\beta((h^{*}_{m})_{m=1}^\infty)=1$. Consequently, $K_{2}(X)=1$.

\end{proof}

\section{Quantifying the Grothendieck property and reflexivity}\label{sect:groth}

The following result is a slight improvement on Lemma \ref{3.2}. For the sake of completeness, we include the proof here.

\begin{lemma}\label{3.3}
For a bounded subset $A$ of $X$ the following are equivalent:
\begin{enumerate}
\item $A$ is relatively weakly compact.
\item Every sequence in $A$ admits a convex block subsequence that is norm convergent.
\item Every sequence in $A$ admits a convex block subsequence that is weakly convergent.
\end{enumerate}
\end{lemma}
\begin{proof}
$(1)\Rightarrow (2)$. Given a sequence $(x_{n})_{n=1}^\infty $ in $A$. Then $(x_{n})_{n=1}^\infty $ admits a subsequence $(y_{n})_{n=1}^\infty $ that is weakly convergent. By Mazur's lemma, $(y_{n})_{n=1}^\infty $ admits a convex block subsequence $(z_{n})_{n=1}^\infty $ that is norm convergent. It follows from Lemma \ref{2.5} that $(z_{n})_{n=1}^\infty $ is a convex block subsequence of $(x_{n})_{n=1}^\infty $.

$(2)\Rightarrow (3)$ is trivial. It remains to prove $(3)\Rightarrow (1)$.

Let $K=\overline{\conv}(A)$. Given any $f\in X^{*}$. We let $c=\sup\limits_{x\in K}\langle f,x\rangle=\sup\limits_{x\in A}\langle f,x\rangle$.
Choose a~sequence $(x_{n})_{n=1}^\infty $ in $A$ so that $\langle f,x_{n}\rangle \rightarrow c$. By the assumption, there exists a sequence $(z_{n})_{n=1}^\infty \in \operatorname{cbs}
((x_{n})_{n=1}^\infty )$ so that $(z_{n})_{n=1}^\infty $ converges weakly to some $x\in K$. It is easy to see that $\langle f,z_{n}\rangle \rightarrow c$. Hence $c=\langle f,x\rangle.$ It follows from James' characterisation of weak compactness via norm-attaining functionals that $K$ is weakly compact and so $A$ is relatively weakly compact.\end{proof}

\begin{proposition}\label{3.11}
A Banach space $X$ has the Grothendieck property if and only if every weak$^{*}$ null sequence in $X^{*}$ admits a convex block subsequence that is norm null.
\end{proposition}
\begin{proof}
The necessity follows from Mazur's lemma. It remains to prove the sufficiency.

Given a weak$^{*}$ null sequence $(f_{n})_{n=1}^\infty$ in $X^{*}$ and any subsequence $(h_{n})_{n=1}^\infty$ of $(f_{n})_{n=1}^\infty$. By the hypothesis, $(h_{n})_{n=1}^\infty$ admits a convex block subsequence $(g_{n})_{n=1}^\infty $ that is norm null. By Lemma \ref{3.3}, the sequence $(f_{n})_{n=1}^\infty$ is relatively weakly compact and hence is weakly null. Thus $X$ has the Grothendieck property.\end{proof}

\begin{definition}
Let $X$ be a Banach space. We set $$G(X)=\sup_{(f_{n})_{n=1}^\infty \subseteq B_{X^{*}} \atop \operatorname{weak}^{*}\textrm{null}}\inf_{(g_{n})_{n=1}^\infty \in \operatorname{cbs}((f_{n})_{n=1}^\infty )}\limsup_{n\to \infty}\|g_{n}\|.$$
\end{definition}

The above-defined quantity measures, in a certain sense, how far is a given Banach space from being a Grothendieck space. This quantification of the Grothendieck property is very different from the one proposed by Kruli\v{s}ov\'a (\cite{Bendova:2014}) who introduced the so-called $\lambda$-\emph{Grothendieck spaces} parametrised by $\lambda\geqslant 1$. Every $\lambda$-Grothendieck space is Grothendieck but not every Grothendieck space is $\lambda$-Grothendieck for some $\lambda \geqslant 1$ (\cite[Theorem 1.2]{Bendova:2014}).

\begin{example}${}$
\begin{enumerate}
\item\label{gex:1} $G(c_{0})=1$,
\item\label{gex:2} $G(\ell_1)=1$,
\item\label{gex:3} $G(C[0,1])=1$.
\end{enumerate}
\end{example}
\begin{proof}
\eqref{gex:1} is clear.

For \eqref{gex:2}, let $(s_{n})_{n=1}^\infty$ be the summing basis  of $c_{0}$, that is, $s_n = \sum\limits_{k=1}^n e_k$ ($n\in \mathbb N$). Then $(s_{\omega}-s_{n})_{n=1}^\infty$ is a weak$^{*}$ null sequence in $B_{\ell_{\infty}}$, where $s_\omega$ is the sequence constantly equal to $1$. It is easy to see that for any $(g_{n})_{n=1}^\infty \in \operatorname{cbs}((s_{\omega}-s_{n})_{n=1}^\infty)$ we have $\|g_{n}\|=1$ ($n\in \mathbb N$). Consequently, $G(\ell_1)=1$.

In order to prove \eqref{gex:3}, for the sake of convenience, we consider $C[-1,1]$ instead. For each $n$, we define
\[h_{n}(t)= \left\{ \begin{array}
                    {r@{\quad\quad}l}

 -\frac{n}{2}, & -\frac{1}{n}\leqslant t<0 \\ \frac{n}{2}, & 0\leqslant t\leqslant \frac{1}{n} \\ 0, & \text{otherwise}
 \end{array} \right. \]
and
\[\varphi(t)= \left\{ \begin{array}
                    {r@{\quad\quad}l}

 -1, & -1\leqslant t<0 \\ 1, & 0\leqslant t\leqslant 1
 \end{array} \right. \]
Let $\nu$ be the Lebesgue measure. A routine argument shows that $\lim\limits_{n\rightarrow\infty}\int fh_{n}\,{\rm d}\nu=0$ for all $f\in C[-1,1]$, which means that $(h_{n})_{n=1}^\infty $ is a weak$^{*}$ null sequence in $B_{C[-1,1]^{*}}$ if we view each $h_{n}\in L_1[-1,1]$ as an element of $C[-1,1]^{*}$. Clearly, $\int \varphi\cdot h_{n}d\nu=1$ for each $n$. Take any $(\nu_{n})_{n=1}^\infty\in \operatorname{cbs}((h_{n})_{n=1}^\infty )$ and write $\nu_{n}=\sum\limits_{i=k_{n-1}+1}^{k_{n}}\lambda_{i}h_{i}$. Then
$$\langle \varphi,\nu_{n}\rangle=\sum\limits_{i=k_{n-1}+1}^{k_{n}}\lambda_{i}\langle \varphi,h_{i}\rangle=1\quad (n\in \mathbb N),$$
which implies that $\|\nu_{n}\|=1$ if we regard $\varphi$ as an element of $B_{C[-1,1]^{**}}$. We have thus proved that $G(C[-1,1])=1$.\end{proof}

We are going to use $G$ to quantify how far is a~given Banach space from being a~Grothendieck space.

\begin{theorem}
A Banach space $X$ has the Grothendieck property if and only if $G(X)=0$.
\end{theorem}
\begin{proof}
The necessary implication follows from Proposition \ref{3.11}.

Suppose that $G(X)=0$. Given a weak$^{*}$ null sequence $(f_{n})_{n=1}^\infty$ in $B_{X^{*}}$, by induction, for each $k$, we get a sequence $(f^{(k)}_{n})_{n=1}^\infty $ so that for all $k=1,2,\ldots,$
\begin{itemize}
    \item $(f^{(1)}_{n})_{n=1}^\infty\in \operatorname{cbs}\big((f_{n})_{n=1}^\infty \big),$
    \item $(f^{(k+1)}_{n})_{n=1}^\infty\in \operatorname{cbs}\big((f^{(k)}_{n})_{n=1}^\infty \big),$
    \item $\limsup\limits_{n\to \infty}\|f^{(k)}_{n}\|<\frac{1}{k}.$
\end{itemize}

For each $n$, we set $h_{n}=f_{n}^{(n)}$. By Lemma \ref{2.5},  $(h_{n})_{n\geqslant k}\in \operatorname{cbs}((f^{(k)}_{n})_{n=1}^\infty )$ for each $k$.
Hence $$\limsup_{n\to \infty}\|h_{n}\|\leqslant \limsup_{n\to \infty}\|f^{(k)}_{n}\|<\tfrac{1}{k}\qquad (k\in \mathbb N).$$
This implies that $(h_{n})_{n=1}^\infty$ is a convex block subsequence of $(f_{n})_{n=1}^\infty$ that converges to $0$ in norm. Again by Proposition \ref{3.11}, $X$ enjoys the Grothendieck property.\end{proof}

\begin{definition}
Let $X$ be a Banach space. We set $$R(X)=\sup_{(x_{n})_{n=1}^\infty \subseteq B_{X}}\inf_{(z_{n})_{n=1}^\infty \in \operatorname{cbs}((x_{n})_{n=1}^\infty )}\operatorname{ca}((z_{n})_{n=1}^\infty ).$$
\end{definition}

\begin{theorem}\label{3.13}
A Banach space $X$ is reflexive if and only if $R(X)=0$.
\end{theorem}
\begin{proof}
The necessity follows from Lemma \ref{3.3}. To prove the sufficiency, we need \cite[Fact 1]{BF}: an operator $T$ from a Banach space $X$ to a Banach space $Y$ is weakly compact if and only if the image under $T$ of every normalised basic sequence in $X$ does not dominate the summing basis $(s_{n})_{n=1}^\infty$ of $c_{0}$. In particular, a Banach space $X$ is reflexive if and only if every normalised basic sequence in $X$ does not dominate the summing basis $(s_{n})_{n=1}^\infty$ of $c_{0}$.

Assume that $X$ is non-reflexive. Then there exists a normalised basic  sequence $(x_{n})_{n=1}^\infty $ in $X$ that dominates the summing basis $(s_{n})_{n=1}^\infty$ in $c_{0}$. That is, for some constant $C>0$, we get $$\|\sum_{i=1}^{n}a_{i}x_{i}\|\geqslant C\|\sum_{i=1}^{n}a_{i}s_{i}\|=C\max_{1\leqslant k\leqslant n}|\sum_{i=k}^{n}a_{i}|,$$ for all $n$ and all scalars $a_{1},a_{2},\ldots,a_{n}.$
By the hypothesis, there exists a sequence $(z_{n})_{n=1}^\infty $ in $\operatorname{cbs}((x_{n})_{n=1}^\infty )$,
$z_{n}=\sum\limits_{i=k_{n-1}+1}^{k_{n}}\lambda_{i}x_{i}$, so that
$\operatorname{ca}((z_{n})_{n=1}^\infty )<C/2$. Thus, for $n\neq m$ we have $\|z_{n}-z_{m}\|<\frac{1}{2}{C}$, yet
$$\|z_{n}-z_{m}\|=\|\sum\limits_{i=k_{n-1}+1}^{k_{n}}\lambda_{i}x_{i}-\sum\limits_{i=k_{m-1}+1}^{k_{m}}\lambda_{i}x_{i}\|\geqslant C
\|\sum\limits_{i=k_{n-1}+1}^{k_{n}}\lambda_{i}s_{i}-\sum\limits_{i=k_{m-1}+1}^{k_{m}}\lambda_{i}s_{i}\|\geqslant C.$$
This contradiction completes the proof.\end{proof}

We discuss the relationship between the quantity $R$ and several commonly used equivalent quantities measuring weak non-compactness.

\begin{theorem}\label{3.15}
Let $X$ be a Banach space. Then $$\operatorname{wck}_{X}(B_{X})\leqslant R(X).$$
\end{theorem}
\begin{proof}
\emph{Case 1.} $X$ is separable.

Let $0<c<\operatorname{wck}_{X}(B_{X})$ be arbitrary. Then there exists a sequence $(x_{n})_{n=1}^\infty $ in $B_{X}$ so that $\textrm{d}(\textrm{clust}_{X^{**}}((x_{n})_{n=1}^\infty ),X)>c$.
Let $\varepsilon>0$. Take any $x^{**}_{0}\in \textrm{clust}_{X^{**}}((x_{n})_{n=1}^\infty )$ and let $d=\textrm{d}(x^{**}_{0},X)$. By the Hahn--Banach theorem, there exists $x^{***}_{0}\in S_{X^{***}}$ so that $\langle x^{***}_{0},x^{**}_{0}\rangle=d$ and $\langle x^{***}_{0},x\rangle=0$ for all $x\in X$. We let
$$C=B_{X^{*}}\cap \{x^{***}\in X^{***}\colon |\langle x^{***},x^{**}_{0}\rangle-d|<\varepsilon\}.$$ By Goldstine's theorem, $x^{***}_{0}\in \overline{C}^{\sigma(X^{***},X^{**})}$. Since $\langle x^{***}_{0},x\rangle=0$ for all $x\in X$, we get $0\in \overline{C}^{\sigma(X^{*},X)}$.
Since $X$ is separable, there exists a weak$^{*}$ null sequence $(f_{m})_{m=1}^\infty$ in $C$. By passing to a subsequence, we may assume that the limit $\lim\limits_{m}\langle x^{**}_{0},f_{m}\rangle$ exists, which is denoted by $a$. By the definition of $C$, $|a-d|\leqslant \varepsilon$.
Since $x^{**}_{0}\in \textrm{clust}_{X^{**}}((x_{n})_{n=1}^\infty )$, we get a subsequence $(y_{n})_{n=1}^\infty $ of $(x_{n})_{n=1}^\infty $ so that $|\langle x^{**}_{0}-y_{n},f_{m}\rangle|<\frac{1}{n}$ for $m=1,2,\ldots,n$. This implies that $\lim\limits_{n\to\infty}\langle f_{m},y_{n}\rangle=\langle x^{**}_{0},f_{m}\rangle$ for each $m$ and then $\lim\limits_{m\to\infty}\lim\limits_{n\to\infty}\langle f_{m},y_{n}\rangle=a$. Given any $(z_{n})_{n=1}^\infty \in \operatorname{cbs}((y_{n})_{n=1}^\infty )$. It is easy to see that $\lim\limits_{m\to\infty}\lim\limits_{n\to\infty}\langle f_{m},z_{n}\rangle=a$.

We \emph{claim} that $|a|\leqslant \operatorname{ca}((z_{n})_{n=1}^\infty )$. Indeed, for any $\delta>0$, we may choose a $N\in \mathbb N$ so that $\|z_{n}-z_{N}\|<\operatorname{ca}((z_{n})_{n=1}^\infty )+\delta$ for all $n\geqslant N$. Then for each $m$ and $n\geqslant N$, we get $$|\langle f_{m},z_{n}\rangle|\leqslant \operatorname{ca}((z_{n})_{n=1}^\infty )+\delta+|\langle f_{m},z_{N}\rangle|.$$
Since $(f_{m})_{m=1}^\infty$ is weak$^{*}$ null, we get, by letting $n\rightarrow \infty$ and $m\rightarrow \infty$, $|a|\leqslant \operatorname{ca}((z_{n})_{n=1}^\infty )+\delta.$ As $\delta$ was arbitrary, the proof of the claim is complete.

It follows that $$c<d\leqslant |a|+\varepsilon \leqslant R(X)+\varepsilon.$$ As $c$ and $\varepsilon$ are arbitrary, we get $\operatorname{wck}_{X}(B_{X})\leqslant R(X).$

\noindent \emph{Case 2}. $X$ is possibly non-separable.

Let $0<c<\operatorname{wck}_{X}(B_{X})$ be arbitrary. Then there exists a sequence $(x_{n})_{n=1}^\infty $ in $B_{X}$ so that $\textrm{d}(\textrm{clust}_{X^{**}}((x_{n})_{n=1}^\infty ),X)>c$.
Let $Y=\overline{\textrm{span}}\{x_{n}\colon n=1,2,\ldots\}$ and $i_{Y}\colon Y\rightarrow X$ be the inclusion map. Since $i_{Y}^{**}\colon Y^{**}\rightarrow X^{**}$ is an isometric embedding, we get $$\textrm{d}(\textrm{clust}_{Y^{**}}((x_{n})_{n=1}^\infty ),Y)\geqslant \textrm{d}(\textrm{clust}_{X^{**}}((x_{n})_{n=1}^\infty ),X)>c.$$
Indeed, let $y^{**}\in \textrm{clust}_{Y^{**}}((x_{n})_{n=1}^\infty )$ and $y\in Y$ be arbitrary. Then $i^{**}_{Y}y^{**}\in \textrm{clust}_{X^{**}}((x_{n})_{n=1}^\infty )$ and
$$\|y^{**}-y\|=\|i^{**}_{Y}y^{**}-y\|\geqslant \textrm{d}(\textrm{clust}_{X^{**}}((x_{n})_{n=1}^\infty ),X).$$
Finally, by Case 1, we get $$c\leqslant \operatorname{wck}_{Y}(B_{Y})\leqslant R(Y)\leqslant R(X).$$ As $c$ was arbitrary, the proof is complete.\end{proof}


\begin{example}${}$
\begin{enumerate}
\item\label{ex3:1} Let $X$ be a Banach space containing a subspace isomorphic to $\ell_1$. Then $R(X)=2$. In particular, $R(\ell_1)=R(C[0,1])=2$.
\item\label{ex3:2} $R(c)=2$, where $c$ denotes the space of all convergent scalar sequences equipped with the supremum norm.
\item\label{ex3:3} $1\leqslant R(c_{0})\leqslant \frac{4}{3}$.\end{enumerate}
\end{example}
\begin{proof}
\eqref{ex3:1}. Let $\varepsilon>0$. By James' distortion theorem, there is a sequence $(x_{n})_{n=1}^\infty $ in $B_{X}$ so that
$\|\sum\limits_{i=1}^{n}a_{i}x_{i}\|\geqslant (1-\varepsilon)\sum\limits_{i=1}^{n}|a_{i}|$ for all $n$ and all scalars $a_{1},a_{2},\ldots,a_{n}.$
For each $(z_{n})_{n=1}^\infty \in \operatorname{cbs}((x_{n})_{n=1}^\infty )$ we write $z_{n}=\sum_{i=k_{n-1}+1}^{k_{n}}\lambda_{i}x_{i}$. Then, for $n<m$, we get
$$\|z_{n}-z_{m}\|=\|\sum\limits_{i=k_{n-1}+1}^{k_{n}}\lambda_{i}x_{i}-\sum\limits_{i=k_{m-1}+1}^{k_{m}}\lambda_{i}x_{i}\|\geqslant 2(1-\varepsilon).$$
This implies that $\operatorname{ca}((z_{n})_{n=1}^\infty )\geqslant 2(1-\varepsilon)$ and hence $R(X)\geqslant 2(1-\varepsilon)$. As $\varepsilon$ was arbitrary, we proved \eqref{ex3:1}.

\eqref{ex3:2}. For each $n$, let
\[x_{n}(i)= \left\{ \begin{array}
                    {r@{\quad}l}

 1, & i\leqslant n \\ -1, & i>n
 \end{array} \right. \]
Given $(z_{n})_{n=1}^\infty \in \operatorname{cbs}((x_{n})_{n=1}^\infty )$, we write $z_{n}=\sum\limits_{i=k_{n-1}+1}^{k_{n}}\lambda_{i}x_{i}$. Then, for $n<m$
$$\sum\limits_{i=k_{n-1}+1}^{k_{n}}\lambda_{i}x_{i}(k_{m-1}+1)=-1, \quad \sum\limits_{i=k_{m-1}+1}^{k_{m}}\lambda_{i}x_{i}(k_{m-1}+1)=1.$$
This implies that $\|z_{n}-z_{m}\|=2$ and so $\operatorname{ca}((z_{n})_{n=1}^\infty )=2$. Thus, we obtain $R(c)=2$.

\eqref{ex3:3}. The inequality $R(c_{0})\geqslant 1$ follows from Theorem \ref{3.15} since for every non-reflexive space $X$ one has $\operatorname{wck}_{X}(B_{X})=1$, which follows for example from \cite[Theorem 1]{GHP} and \cite[Proposition 2.2]{CKS}. The inequality $R(c_{0})\leqslant 4/3$ was pointed out by W. B. Johnson; we present it here with his permission.

Suppose that $(x_{n})_{n=1}^\infty $ is a sequence in $B_{c_{0}}$. By passing to a subsequence, we may assume that $(x_{n})_{n=1}^\infty $ converges coordinate-wise to some $x\in B_{\ell_{\infty}}$. By passing to further subsequence and making a small perturbation we may assume that there are $k_{1}<k_{2}<\ldots$ so that $x_{n}$ is supported on $\{1,2,\ldots,k_{n}\}$ and $x_{n+1}(i)=x(i),i=1,2,\ldots,k_{n}$. We define $z_{n}=\frac{2}{3}x_{2n}+\frac{1}{3}x_{2n+1}$ ($n\in \mathbb N$).

We \emph{claim} that $\|z_{n}-z_{m}\|\leqslant \frac{4}{3}$ for all $n,m$, $m>n$. Indeed,

\[|z_{n}(i)-z_{m}(i)|=\left\{ \begin{array}
                    {r@{\quad\quad}l}

 |\frac{2}{3}x_{2n}(i)+\frac{1}{3}x(i)-\frac{2}{3}x_{2m}(i)-\frac{1}{3}x(i)|\leqslant \frac{4}{3}, & i\leqslant k_{2n} \\
 |\frac{1}{3}x_{2n+1}(i)-\frac{2}{3}x_{2m}(i)-\frac{1}{3}x(i)|\leqslant \frac{4}{3}, & k_{2n}<i\leqslant k_{2n+1} \\
 |-\frac{2}{3}x_{2m}(i)-\frac{1}{3}x(i)|\leqslant 1, & k_{2n+1}<i\leqslant k_{2m} \\
 |-\frac{1}{3}x_{2m+1}(i)|\leqslant \frac{1}{3}, & k_{2m}<i\leqslant k_{2m+1} \\
 \end{array} \right. \]
Consequently, $\operatorname{ca}((z_{n})_{n=1}^\infty )\leqslant \frac{4}{3}$ and the proof is completed. \end{proof}

We require an elementary lemma whose proof is straightforward.

\begin{lemma}\label{3.12}
Suppose that $(f_{n})_{n=1}^\infty$ is a weak$^{*}$ null sequence in $X^{*}$. Then
$$\limsup\limits_{n\to\infty}\|f_{n}\|\leqslant\operatorname{ca}((f_{n})_{n=1}^\infty )\leqslant 2\limsup\limits_{n\to\infty}\|f_{n}\|.$$
\end{lemma}

An immediate consequence of Lemma \ref{3.12} is the following quantification of implications ($\star$).

\begin{theorem}\label{3.4}
Let $X$ be a Banach space. Then $$K_{1}(X)\leqslant G(X)\leqslant R(X^{*}).$$
\end{theorem}

In order to quantify ($\star\star$), we require a lemma.

\begin{lemma}\label{3.14}
Let $X$ be a Banach space containing no subspaces isomorphic to $\ell_1$. Suppose that $f_{n}\stackrel{\operatorname{weak}^{*}}{\longrightarrow}0$ in $X^{*}$. Then
$$\limsup\limits_{n\to\infty}\|f_{n}\|\leqslant 2\beta\big((f_{n})_{n=1}^\infty\big).$$
\end{lemma}
\begin{proof}
Let $0<c<\limsup\limits_{n\to\infty}\|f_{n}\|$. By passing to a subsequence, we may assume that $\|f_{n}\|>c$ for all $n$. Choose $x_{n}\in B_{X}$ with $\langle f_{n},x_{n}\rangle>c$ ($n\in\mathbb N$). Passing to a further subsequence, by Rosenthal's $\ell_1$-theorem, we may assume that $(x_{n})_{n=1}^\infty $ is weakly Cauchy.
Let $\varepsilon>0$. Since $f_{n}\stackrel{\operatorname{weak}^{*}}{\longrightarrow}0$, we obtain, by induction, a strictly increasing sequence $(k_{n})_{n=1}^\infty$ of even integers so that $\langle f_{k_{n}},x_{k_{n}}-x_{2n-1}\rangle>c-\varepsilon$ for all $n$. We set $y_{n}=\frac{1}{2}(x_{k_{n}}-x_{2n-1})$. Then $(y_{n})_{n=1}^\infty $ is weakly null in $B_{X}$. Let us define a weakly null sequence $(z_{n})_{n=1}^\infty $ in $B_{X}$ by $z_{k_{n}}=y_{n}$ and $0$ otherwise. Then
$$\beta\big((f_{n})_{n=1}^\infty \big)\geqslant \limsup\limits_{n\to\infty}|\langle f_{n},z_{n}\rangle|\geqslant\limsup\limits_{n\to\infty}|\langle f_{k_{n}},z_{k_{n}}\rangle|
\geqslant \frac{c-\varepsilon}{2}.$$
Letting $\varepsilon\rightarrow 0$, we get $\beta\big((f_{n})_{n=1}^\infty \big)\geqslant \frac{c}{2}.$ As $c$ was arbitrary, the proof is complete.\end{proof}

\begin{theorem}
Let $X$ be a Banach space containing no subspaces isomorphic to $\ell_1$. Then $$K_{1}(X)\leqslant G(X)\leqslant R(X^{*})\leqslant 8K_{2}(X).$$
\end{theorem}
\begin{proof}
By Theorem \ref{3.4}, it suffices to prove the inequality $R(X^{*})\leqslant 8K_{2}(X).$

Let $0<c<R(X^{*})$. Then there exists a sequence $(f_{n})_{n=1}^\infty$ in $B_{X^{*}}$ so that
$$\operatorname{ca}\big((g_{n})_{n=1}^\infty\big)>c \qquad \Big( (g_{n})_{n=1}^\infty\in \operatorname{cbs}\big((f_{n})_{n=1}^\infty \big)\Big).$$
Since $X$ contains no isomorphic copy of $\ell_1$, it follows from \cite[Proposition 3.11]{Bo} (\emph{cf}. \cite[Proposition 11]{Pf}) that $B_{X^{*}}$ is \emph{weak$^{*}$ convex block compact}, that is, every sequence in $B_{X^{*}}$ admits a weak$^{*}$ convergent convex block subsequence. By passing to a convex block subsequence, by Lemma \ref{2.5} we may assume that $f_{n}\stackrel{\operatorname{weak}^{*}}{\longrightarrow}f$ for some $f\in B_{X^{*}}$. Hence, we get
$$\operatorname{ca}((g_{n})_{n=1}^\infty)>c \qquad \Big((g_{n})_{n=1}^\infty\in \operatorname{cbs}\big((f_{n}-f)_{n=1}^\infty\big)\Big).$$
Rescaling if necessary, we may assume that $(f_{n})_{n=1}^\infty$ is a weak$^{*}$ null sequence in $B_{X^{*}}$ and
$$\operatorname{ca}\big((g_{n})_{n=1}^\infty\big)>\frac{c}{2} \qquad \Big((g_{n})_{n=1}^\infty\in \operatorname{cbs}\big((f_{n})_{n=1}^\infty \big)\Big).$$
By Lemma \ref{3.12} and Lemma \ref{3.14}, we arrive at
$$\frac{c}{2}<\operatorname{ca}\big((g_{n})_{n=1}^\infty\big)\leqslant 2\limsup\limits_{n}\|g_{n}\|\leqslant 4\beta\big((g_{n})_{n=1}^\infty\big) \qquad \Big( (g_{n})_{n=1}^\infty\in \operatorname{cbs}\big((f_{n})_{n=1}^\infty \big)\Big).$$ This implies that $K_{2}(X)\geqslant \frac{c}{8}$. Since $c$ was arbitrary, the proof is complete. \end{proof}

\section{Quantifying property ($\mu^{s}$)}\label{sect:mus}

 For a bounded sequence $(f_{n})_{n=1}^\infty$ in $X^{*}$, we define $$c\alpha\big((f_{n})_{n=1}^\infty \big)=\alpha\big((\frac{1}{n}\sum_{i=1}^{n}f_{i})_{n=1}^\infty\big).$$
Then $c\alpha((f_{n})_{n=1}^\infty )=0$ if and only if $(f_{n})_{n=1}^\infty$ is Ces\`{a}ro convergent to $0$ with respect to $\mu(X^{*},X)$. A direct argument shows that $
c\alpha((f_{n})_{n=1}^\infty )=c\alpha((f_{n})_{n\geqslant N+1})$ for every positive integer $N$.

\begin{definition}
Let $X$ be a Banach space. We set
$$\mu^{s}(X)=\sup_{(f_{n})_{n=1}^\infty \subseteq B_{X^{*}} \atop \operatorname{weak}^{*}\textrm{null}}\inf_{(g_{n})_{n=1}^\infty\in \textrm{ss}((f_{n})_{n=1}^\infty )}\sup_{(h_{n})_{n=1}^\infty \in
\textrm{ss}((g_{n})_{n=1}^\infty)}c\alpha\big((h_{n})_{n=1}^\infty\big).$$
\end{definition}

\begin{theorem}\label{3.8}
A Banach space $X$ has property $(\mu^{s})$ if and only if $\mu^{s}(X)=0$.
\end{theorem}
\begin{proof}
The sufficient part is trivial. We only prove the necessary part.

Given a weak$^{*}$ null sequence $(f_{n})_{n=1}^\infty$ in $B_{X^{*}}$, by induction, for each $k$ we may find a~sequence $((g_{n})^{(k)})_{n=1}^\infty$ in $X^{*}$ such that
\begin{itemize}
    \item $((g_{n})^{(1)})_{n=1}^\infty\in \textrm{ss}((f_{n})_{n=1}^\infty ),$
    \item $((g_{n})^{(k+1)})_{n}\in \textrm{ss}(((g_{n})^{(k)})_{n}),$
    \item $c\alpha\big((g_{n})_{n=1}^\infty\big)<\frac{1}{k} \quad \big((g_{n})_{n=1}^\infty \in \textrm{ss}(((g_{n})^{(k)})_{n})\big).$
\end{itemize}

Let $g_{n}=(g_{n})^{(n)}$ $(n=1,2,\ldots)$. Then $(g_{n})_{n=1}^\infty$ is a subsequence of $(f_{n})_{n=1}^\infty$. Take any subsequence $(h_{n})_{n=1}^\infty $ of
$(g_{n})_{n=1}^\infty$. By construction, for each $k$, there exists $N_{k}\in\mathbb N$ so that $(h_{n})_{n\geqslant N_{k}+1}\in \textrm{ss}(((g_{n})^{(k)})_{n}).$
Consequently, $$c\alpha\big((h_{n})_{n=1}^\infty\big)=c\alpha((h_{n})_{n\geqslant N_{k}+1})<\frac{1}{k}.$$
As $k$ was arbitrary, $c\alpha\big((h_{n})_{n=1}^\infty\big)=0$. Thus the sequence $(h_{n})_{n=1}^\infty $ is Ces\`{a}ro convergent to $0$ with respect to $\mu(X^{*},X)$, which completes the proof. \end{proof}

\begin{definition}For a Banach space $X$, we set
$$\operatorname{bs}(X)=\sup_{(x_{n})_{n=1}^\infty \subseteq B_{X}}\inf_{(y_{n})_{n=1}^\infty \in \textrm{ss}((x_{n})_{n=1}^\infty )}\sup_{(z_{n})_{n=1}^\infty \in \textrm{ss}((y_{n})_{n=1}^\infty )}
\textrm{cca}((z_{n})_{n=1}^\infty ).$$\end{definition}

Clearly, $\operatorname{bs}(B_{X})\leqslant \operatorname{bs}(X)$. Combining Theorem \ref{1.1} with \cite[Corollary 4.3]{BKS}, we see that $\operatorname{bs}(X)=0$ if and only if $X$ has the Banach--Saks property.

\begin{theorem}
Let $X$ be a Banach space. Then $$\frac{1}{3}K_{1}(X)\leqslant \mu^{s}(X)\leqslant \operatorname{bs}(X^{*}).$$
\end{theorem}
\begin{proof}
The latter inequality follows from Lemma \ref{2.3}, so it remains to prove only the former one.

Let $0<c<K_{1}(X)$ and let $\varepsilon>0$. Then there exist a weak$^{*}$ null sequence $(f_{n})_{n=1}^\infty$ in $B_{X^{*}}$ and a subsequence $(g_{n})_{n=1}^\infty$ of $(f_{n})_{n=1}^\infty$ such that
\begin{itemize}
    \item $\alpha\big((h_{n})_{n=1}^\infty\big)>c \quad \big((h_{n})_{n=1}^\infty \in \operatorname{cbs}((f_{n})_{n=1}^\infty )\big),$
    \item $c\alpha((g_{n})_{n=1}^\infty)<\mu^{s}(X)+\varepsilon.$
\end{itemize}
It follows from Lemma \ref{2.5} that
\begin{align*}
c&<\alpha\big((\frac{1}{2^{n-1}}\sum_{i=2^{n-1}+1}^{2^{n}}g_{i})_{n=1}^\infty\big)\\
&\leqslant 2\alpha\big((\frac{1}{2^{n}}\sum_{i=1}^{2^{n}}g_{i})_{n=1}^\infty\big)+\alpha\big((\frac{1}{2^{n-1}}\sum_{i=1}^{2^{n-1}}g_{i})_{n=1}^\infty\big)\\
&\leqslant 3\alpha\big((\frac{1}{n}\sum_{i=1}^{n}g_{i})_{n=1}^\infty\big)\\
&\leqslant 3\mu^{s}(X)+3\varepsilon.
\end{align*}
As $c$ and $\varepsilon$ are arbitrary, we arrive at $K_{1}(X)\leqslant 3\mu^{s}(X)$, which completes the proof. \end{proof}

Finally, we present a result that directly quantifies  \cite[Proposition 2.2]{Ro}.

\begin{theorem}
Let $X$ be a reflexive space. Then $$\mu^{s}(X)\leqslant \operatorname{bs}(X^{*})\leqslant 4\mu^{s}(X).$$
\end{theorem}
\begin{proof}
The former inequality follows from Lemma \ref{2.3}, so we need to prove the latter one.

Let $0<c<\operatorname{bs}(X^{*})$. Then there exists a sequence $(f_{n})_{n=1}^\infty$ in $B_{X^{*}}$ so that
\begin{equation}\label{4.1}
\sup_{(h_{n})_{n=1}^\infty \in \textrm{ss}((g_{n})_{n=1}^\infty)}\textrm{cca}((h_{n})_{n=1}^\infty )>c \quad \big( (g_{n})_{n=1}^\infty\in \textrm{ss}((f_{n})_{n=1}^\infty )\big).
\end{equation}
Due to reflexivity, we may assume that $f_{n}\stackrel{\operatorname{weak}^{*}}{\longrightarrow} f$ for some $f\in B_{X^{*}}$.

Given any $(g_{n})_{n=1}^\infty\in \textrm{ss}((\frac{f_{n}-f}{2})_{n=1}^\infty),$ by \eqref{4.1}, there exists a subsequence $(h_{n})_{n=1}^\infty $ of $(2g_{n}+f)_{n=1}^\infty$ such that $\textrm{cca}((h_{n})_{n=1}^\infty )>c.$ Again, by reflexivity of $X$, we get
$$2c\alpha\big((h_{n}-f)_{n=1}^\infty\big)\geqslant \textrm{cca}((h_{n}-f)_{n=1}^\infty)=\textrm{cca}((g_{n})_{n=1}^\infty )>c.$$
As $(\frac{h_{n}-f}{2})_{n=1}^\infty$ is a subsequence of $(g_{n})_{n=1}^\infty$, $\mu^{s}(X)\geqslant \frac{c}{4}.$ Since $c$ was arbitrary, the proof is complete. \end{proof}

\subsection*{Acknowledgements}
The first-named author would like to thank W.~B. Johnson and B.~Wallis for helpful discussions and comments. The second-named author is indebted to G.~Plebanek for making available \cite{FP} to us.


\begin{thebibliography}{MMM}
\frenchspacing


\bibitem[AK]{AK}
F. Albiac and N. J. Kalton,
{\em Topics in Banach space theory}, Springer, 2005.


\bibitem[AC]{AC}
C. Angosto and B. Cascales, {\em Measures of weak non-compactness in Bananch spaces}, Topology Appl. {\bf 156} (2009), 1412--1421.



\bibitem[AR]{AR}
A. Avil\'{e}s and J. Rodr\'{\i}guez, {\em Convex combinations of weak*-convergent sequences and the Mackey topology}, Mediterr. J. Math. {\bf 13} (2016), 4995--5007.

\bibitem[Ba]{Ba}
A. Baernstein II, {\em On reflexivity and summability}, Studia Math. {\bf 42} (1972), 91--94.


\bibitem[BS]{BS}
S.~Banach and S.~Saks, {\em Sur la convergence forte dans les champs $L_p$}, Studia Math. {\bf 2} (1930), 51--57.


\bibitem[BF]{BF}
K.~Beanland and D.~Freeman, {\em Ordinal ranks on weakly compact and Rosenthal operators}, Extracta Math. {\bf 26} (2011), 173--194.

\bibitem[Be]{Bendova:2014} H.~Bendov\'a, \emph{Quantitative Grothendieck property}. J. Math. Anal. Appl. {\bf 412} (2014), 1097--1104.

\bibitem[BKS]{BKS}
H. Bendov\'{a}, O. F. K. Kalenda, and J. Spurn\'{y}, {\em Quantification of the Banach--Saks property}, J. Funct. Anal. {\bf 268} (2015), 1733--1754.

\bibitem[Bo]{Bo}
J. Bourgain, {\em La propri\'{e}t\'{e} de Radon-Nikodym}, In: Publications Math\'{e}matiques del'Universit\'{e} Pierre et Marie Curie Paris VI, Cours de troisi\`{e}me cycle, n. {\bf 36} (1979).


\bibitem[CKS]{CKS}
B. Cascales, O. F. K. Kalenda, and J. Spurn\'{y}, {\em A quantitative version of James's compactness theorem}, Proc. Edinb. Math. Soc. {\bf 55} (2012), 369--386.


\bibitem[DDS]{DDS}
B.~De Pagter, P.~G.~Dodds, and F.~A.~Sukochev, {\em On weak* convergence sequences in duals of symmetric spaces of $\tau$-measurable operators}, Israel J. Math. {\bf 222} (2017), 125--164.


\bibitem[DF]{DF}
S.~D\'{i}az and A.~Fern\'{a}ndez, {\em Reflexivity in Banach lattices}, Arch Math. {\bf 63} (1994), 549--552.



\bibitem[D]{D}
J.~Diestel, {\em Sequences and Series in Banach Spaces}, Springer, New York, 1984.


\bibitem[DRS]{DRS}
J.~Diestel, W.~M. Ruess, and W.~Schachermayer, {\em Weak compact in $L^{1}(\mu,X)$}, Proc. Amer. Math. Soc. {\bf 118} (1993), 447--453.


\bibitem[DRT]{DRT}
P.~N.~Dowling, N.~Randrianantoanina, and B.~Turett, {\em Remarks on James's distortion theorems II}, Bull. Austral. Math. Soc. {\bf 59}(1999), 515--522.


\bibitem[EM]{EM}
P.~Erd\H{o}s and M.~Magidor, {\em A note on regular methods of summability and the Banach--Saks property}, Proc. Amer. Math. Soc. {\bf 59} (1976), 232--234.


\bibitem[FJP]{FJP}
T.~Figiel, W.~B. Johnson, and A.~Pe{\l}czy\'{n}ski, {\em Some approximation properties of Banach spaces and Banach lattices}, Israel J. Math. {\bf 183} (2011), 199--231.

\bibitem[FP]{FP} R.~Frankiewicz and G.~Plebanek, \emph{Convex combinations and weak null sequences}, Bull. Pol. Acad. Sci., Math. \textbf{45}, No. 3 (1997), 221--225.

\bibitem[GK]{GK} M.~Gonz\'alez and T.~Kania, Grothendieck spaces: the landscape and perspectives, preprint (2021), 55 pp.

\bibitem[GHP]{GHP}
A.~S.~Granero, J.~M.~Hern\'{a}ndez, and H.~Pfitzner, {\em The distance $\textrm{dist}(\mathcal{B},X)$ when $\mathcal{B}$ is a boundary of $B(X^{**})$}, Proc. Amer. Math. Soc. {\bf 139} (2011), 1095--1098.

\bibitem[Ja]{Ja}
H. Jarchow, \emph{Locally Convex Spaces}, B.G. Teubner, Stuttgart, 1981.

\bibitem[Jo]{Jo}
W. B. Johnson, {\em Extensions of $c_{0}$}, Positivity. {\bf 1} (1997), 55--74.


\bibitem[KKS]{KKS}
M. Ka\v{c}ena, O. F. K. Kalenda, and J. Spurn\'{y}, {\em Quantitative Dunford--Pettis property}, Adv. Math. {\bf 234} (2013), 488--527.


\bibitem[Ka]{Ka}
S. Kakutani, {\em Weak convergence in uniformly convex spaces}, Tohoku Math. J. {\bf 45} (1938), 188-193.



\bibitem[KPS]{KPS}
O. F. K. Kalenda, H. Pfitzner, and J. Spurn\'{y}, {\em On quantification of weak sequential completeness},
J. Funct. Anal. {\bf 260} (2011), 2986--2996.


\bibitem[KS]{KS}
O. F. K. Kalenda and J. Spurn\'{y}, {\em Quantification of the reciprocal Dunford-Pettis property}, Studia Math. {\bf 210} (2012), 261--278.

\bibitem[KS1]{KalendaSupurny:2012} O.~F.~K.~Kalenda and J.~Spurný, On a difference between quantitative weak sequential completeness and the quantitative Schur property. Proc. Am. Math. Soc. \textbf{140}, No. 10 (2012), 3435--3444.

\bibitem[KP]{KP}
N. J. Kalton and A. Pe{\l}czy\'{n}ski, {\em Kernels of surjections from $\mathcal{L}_{1}$-spaces with an application to Sidon sets}, Math. Ann. {\bf 309} (1997), 135-158.

\bibitem[Kr]{Krulisova:2017} H.~Kruli\v{s}ov\'a, \emph{
$C^*$-algebras have a quantitative version of Pe{\l}czy\'nski’s property $(V)$}, Czech. Math. J. \textbf{67}, No. 4 (2017), 937--951.

\bibitem[LT]{LT} J. Lindenstrauss and L. Tzafriri,
{\em Classical Banach Spaces I, Sequence Spaces}, Springer, Berlin, 1977.


\bibitem[LART]{LART}
J. Lopez-Abad, C. Ruiz, and P. Tradacete, {\em The convex hull of a Banach--Saks set}, J. Funct. Anal. {\bf 266} (2014), 2251--2280.



\bibitem[NW]{NW}
T. Nishiura and D. Waterman, {\em Reflexivity and summability}, Studia Math. {\bf 23} (1963), 53--57.


\bibitem[Pf]{Pf}
H. Pfitzner, {\em Boundaries for Banach spaces determine weak compactness}, Invent. Math. {\bf 182} (2010), 585--604.


\bibitem[Ro]{Ro}
J. Rodr\'{\i}guez, {\em Ces\`{a}ro convergent sequences in the Mackey topology}, Mediterr. J. Math. \textbf{16}, No. 5, Paper No. 117 (2019), 19 pp.


\bibitem[SW]{SW}
G. Schl\"{u}chtermann and R. F. Wheeler, {\em On strongly WCG Banach spaces}, Math. Z. {\bf 199} (1988), 387--398.

\bibitem[U]{U}
A. \"{U}lger, {\em Weak compactness in $L^{1}(\mu,X)$}, Proc. Amer. Math. Soc. {\bf 113} (1991), 143--149.

\end{thebibliography}
\end{document}